\documentclass[11pt]{article}

\usepackage{amssymb}
\usepackage{amsmath}
\usepackage{amsthm}
\usepackage{amscd}
\usepackage{graphicx}
\usepackage{amsfonts}
\usepackage{xcolor}

\usepackage{hyperref}
\hypersetup{linktocpage}

\hypersetup{colorlinks,
    linkcolor={red!50!black},
    citecolor={blue!80!black},
    urlcolor={blue!80!black}}
\usepackage{float}

\newtheorem{thm}{Theorem}
\newtheorem{cor}[thm]{Corollary}
\newtheorem{lem}[thm]{Lemma}

\newtheorem{prop}[thm]{Proposition}
\newtheorem{prob}[thm]{Problem}

\theoremstyle{definition}
\newtheorem{defn}[thm]{Definition}
\theoremstyle{remark}

\renewcommand{\phi }{\varphi}

\renewcommand{\ll }{\left\langle\hspace{-.7mm}\left\langle }
\newcommand{\rr }{\right\rangle\hspace{-.7mm}\right\rangle }
\renewcommand{\d }{{\rm def} }

\newcommand{\ZZ }{{\mathbb Z}}

\newcommand{\rk}{\operatorname{rk}}

\begin{document}

\title{A simple construction of finitely generated infinite torsion groups}

\author{D. Osin\thanks{This work has been supported by the NSF grant DMS-1853989}}
\date{}
\maketitle
\begin{abstract}
The goal of this note is to provide yet another proof of the following theorem of Golod: there exists an infinite finitely generated group $G$ such that every element of $G$ has finite order. Our proof is based on the Nielsen-Schreier index formula and is simple enough to be included  in a standard group theory course.
\end{abstract}

\section{Introduction}

A group $G$ is said to be \emph{torsion} (or \emph{periodic}) if every element of $G$ has finite order. Obviously, every finite group has this property. Infinite torsion groups can be constructed as direct products of finite groups; however, such groups are not finitely generated. The following famous problem was posed by William Burnside in 1902 and served as a catalyst for research in group theory throughout the 20th century.

\begin{prob}
Is every finitely generated torsion group finite?
\end{prob}

It is easy to show that the answer is affirmative for abelian groups. By induction, this generalizes to all solvable groups. In 1911, Schur proved that every finitely generated torsion subgroup of $GL(n,\mathbb C)$ is finite \cite{S}. The same result holds for linear groups over arbitrary fields. Indeed, the Burnside problem for such groups can be reduced to the case of solvable groups by utilizing the well-known Tits alternative: \emph{every finitely generated linear group contains either a solvable subgroup of finite index or a non-cyclic free subgroup} \cite{Tit}.

Despite these positive results, the answer to Burnside's question turns out to be negative in general as demonstrated by Golod \cite{Gol} in 1964.

\begin{thm}[Golod] \label{main}
There exists a finitely generated infinite torsion group.
\end{thm}

The proof of Theorem \ref{main} given in \cite{Gol} relied on a sufficient condition for certain graded algebras to be infinite dimensional, known as the Golod-Shafarevich inequality \cite{GS} (for a simplified version of Golod's argument, see \cite{Ols95}). Numerous other constructions of infinite finitely generated torsion groups have been discovered since then. Notable examples include groups generated by finite automata \cite{Ale}, groups of interval exchange transformations \cite{Gri}, inductive limits of hyperbolic groups \cite{Gro,Ols93}, and certain groups of finite exponent. The latter class of examples deserves a more detailed discussion.

Recall that a group $G$ has \emph{exponent} $n\in \mathbb N$ if every element $g\in G$ satisfies the identity $g^n=1$ and $n$ is the smallest natural number with this property. It is well-known and easy to prove that every group of exponent $2$ is abelian. Therefore, finitely generated groups of exponent $2$ are finite. Burnside \cite{Bur} proved that finitely generated groups of exponent $3$ are finite. Sanov \cite{San} and Hall \cite{Hal} obtained the same result for exponents $4$ and $6$, respectively. However, it remains an open problem whether a finitely generated group of exponent $5$ can be infinite.

As we move on to larger exponents, we encounter infinite finitely generated groups. More precisely, let $B(m,n)$ denote the free group of rank $m$ in the variety of all groups of exponent~$n$. In a monumental series of papers \cite{NA1,NA2,NA3}, Novikov and Adian showed that $B(m,n)$ is infinite for all $m\ge 2$ and all odd $n\ge 4381$. An improved version of the original proof for odd $n\ge 665$ can be found in the book \cite{A}. In \cite{Ols83}, Olshanskii suggested a much simpler geometric proof of the Novikov-Adian theorem for odd $n>10^{10}$. The question of whether $B(m,n)$ is infinite for all $m\ge 2$ and all sufficiently large even $n$ remained open until the mid-1990s, when Ivanov \cite{Iva} and Lysenok \cite{Lys} independently gave the affirmative answer.

Despite the progress made, there is still no elementary construction of a finitely generated, infinite group of finite exponent. The easiest proof, given by Olshanskii in \cite{Ols83}, is approximately $30$ pages long and rather technical. On the other hand, the constructions of Aleshin \cite{Ale} and Grigorchuk \cite{Gri}, as well as Olshanskii's version of Golod's proof \cite{Ols95}, are simple enough to be discussed in a standard group theory course (in these examples, the groups are not of finite exponent).

The goal of this paper is to provide yet another elementary proof of Golod's theorem that only relies on basic properties of finitely generated abelian groups and the  Nielsen-Schreier formula. In one form or another, this proof appeared in \cite{LO,OO,Osi,SP}. However, the exposition in these papers was ``spoiled" by technicalities necessary to ensure certain additional properties. Here, we present the argument in its simplest form.

\paragraph{Acknowledgments.} The authors is grateful to Alexander Olshanskii and Pierre de la Harpe for useful remarks and suggestions.

\section{Deficiency of finitely presented groups}

To make our paper as self-contained as possible, we review the standard results about abelian groups involved in the proof. Our primary reference is \cite{KM}, although these results can be found in most group theory textbooks.

Recall that every finitely generated abelian group $A$ can be decomposed as a direct sum of cyclic groups \cite[Theorem 8.1.2]{KM}.  We denote the number of infinite summands in this decomposition by $\rk(A)$. If $A$ is free abelian, $\rk(A)$ is the usual rank of $A$. Clearly, $\rk(A)$ serves as the lower bound on the number of generators of $A$.

The standard proof of the decomposition theorem for finitely generated abelian groups goes through establishing the following fact (see \cite[Theorem 8.1.1]{KM}).

 \emph{For any finitely generated free abelian group $B$ and any subgroup $C\le B$, there exist bases $\{b_1, \ldots, b_n\}$ and $\{c_1, \ldots, c_m\}$ ($m\le n$) of $B$ and $C$, respectively, and positive integers $d_1, \ldots, d_m$ such that $c_i=b_i^{d_i}$ for all $1\le i\le m$}.

In this notation, the quotient $A=B/C$ decomposes as $A=\ZZ^{n-m} \oplus \ZZ/d_1\ZZ \oplus \ldots \oplus \ZZ /d_m\ZZ$.
In particular, we have $$\rk(A)=n-m=\rk(B)-\rk(C).$$

A group presentation $G=\langle X\mid \mathcal R\rangle$ is said to be \emph{finite} if the set of generators $X$ and the set of relations $\mathcal R$ are finite. Given two elements $x,y$ of a group $G$, we  write $x^y$ for $y^{-1}xy$. Given a subset $S\subseteq G$, we denote by $\ll S\rr ^G$ the \emph{normal closure} of $S$ in $G$. That is, $\ll S\rr^G$ is the subgroup of $G$ generated by the set $\{ s^g\mid s\in S, \; g\in G\}$.

\begin{lem}\label{Zd}
Let $G$ be a group given by a finite presentation $\langle X\mid \mathcal R\rangle$ and let $d=|X|-|\mathcal R|$. Suppose that $d>0$. Then there exists a surjective homomorphism $G\to \ZZ^{d}$.
\end{lem}

\begin{proof}
Let $F$ denote the free group with the basis $X$. The group $G=F/\ll \mathcal R\rr^F$ surjects onto the abelian group $$A=F\big/\ll \mathcal R\rr^F[F,F]\cong (F/[F,F])\Big/ \Big(\ll \mathcal R\rr^F[F,F]/[F,F]\Big).$$ The group $B=F/[F,F]$ is free abelian of rank $|X|$ (see \cite[p.99]{KM}). It is easy to see that the group $C=\ll \mathcal R\rr^F[F,F]/[F,F]$ is generated by the natural image of $\mathcal R$ in $C$. This implies that $\rk(C)\le |\mathcal R|$. Therefore, $\rk(A)= \rk(B)-\rk(C)\ge d$ and the existence of the required surjection follows from the definition of $\rk(A)$.
\end{proof}

\begin{cor}\label{max}
For any finitely presented group $G$, the difference $|X|-|\mathcal R|$ is uniformly bounded from above over all finite presentations $\langle X\mid \mathcal R\rangle$ of $G$.
\end{cor}
\begin{proof}
If $G$ can be generated by $n$ elements, then $|X|-|\mathcal R|\le n$ for any finite presentation $\langle X\mid \mathcal R\rangle$ of $G$ by Lemma \ref{Zd}.
\end{proof}

Corollary \ref{max} allows us to formulate the following.

\begin{defn}
For a finitely presented group $G$, the maximum of the difference between the number of generators and the number of relations over all finite presentations of $G$ is called the \emph{deficiency} of $G$ and denoted by $\d (G)$.
\end{defn}

Recall that a subset $T$ of a group $G$ is a \emph{left} (respectively, \emph{right}) \emph{transversal} of a subgroup $H\le G$ if $G=\bigsqcup_{t\in T} tH$ (respectively, $G=\bigsqcup_{t\in T} Ht$); clearly, we have $|T|=|G:H|$ in both cases.

\begin{lem}\label{sm}
Let $G$ be a finitely presented group. Every finite index subgroup $H\le G$ is finitely presented and we have
$\d (H)-1\ge (\d (G)-1)|G:H|.$
\end{lem}
\begin{proof}
Let $G=F/R$, where $F$ is free of rank $r$, $R=\ll R_1, \ldots , R_s\rr ^F$, and $r-s=\d(G)$. Let $H$ be a finite index subgroup of $G$, $K$ the full preimage of $H$ in $F$. By the Nielsen-Schreier formula, $K$ is free of rank $(r-1)j+1$, where $j=|F:K|=|G:H|$.

Let $T$ be a left transversal of $K$ in $F$. For every element $f\in F$ we have $f=tk$ for some $t\in T$ and $k\in K$. Hence,   $R_i^f=(R_i^t)^k\in \ll R_i^t\rr ^K$ for all $i=1, \ldots, s$. This easily implies $$R=\left\langle \left\{ R_i^f \mid i=1, \ldots, s,\, f\in F\right\}\right\rangle=\ll \Big\{ R_i^t \mid i=1, \ldots, s,\, t\in T\Big\} \rr ^K.$$ Thus, the group $H=K/R$ has a presentation with $(r-1)j+1$ generators and $s|T|=sj$ relations. Therefore, we have
$$
\d (H)-1\ge (r-1)j - sj = (r-s-1)j=(\d(G)-1)|G:H|.
$$
\end{proof}

\section{Proof of Golod's theorem}

Let $\mathcal D$ denote the class of all finitely presented groups that contain a finite index normal subgroup of deficiency at least $2$. For a group $G$, we denote by $\widehat G$ the quotient of $G$ by the intersection of all finite index subgroups of $G$.

The idea of the proof of the following proposition is borrowed from \cite{BP}.

\begin{prop}\label{quot}
Let $G\in \mathcal D$. For every $g\in G$, there exists $m\in \mathbb Z$ such that, for every $\ell\in \mathbb N$, the quotient group $Q=G/\ll g^{\ell m}\rr ^G$ belongs to $\mathcal D$ and the image of $g$ in $\widehat Q$ has finite order.
\end{prop}

\begin{proof}
If the image of $g$ in $\widehat G$ already has finite order, we can take $m=0$. Henceforth, we assume that the order of the image of $g$ in $\widehat G$ is infinite. That is, for any $i\in \mathbb N$, there exists a finite index subgroups $K\lhd G$ such that $|\langle g\rangle K / K |> i$.

Let $M$ be a finite index normal subgroup of $G$ such that $\d (M)\ge 2$. By our assumption, we can find a finite index subgroup $K\lhd G$ such that $|\langle g\rangle K / K |> |G/M|$. Let $N=K\cap M\lhd G$. Clearly, we have
\begin{equation}\label{o(g)}
|\langle g\rangle N / N |\ge |\langle g\rangle K / K | > |G/M|.
\end{equation}

Set $m=|\langle g\rangle N / N |$ and $f=g^{\ell m}$, where $\ell $ is an arbitrary natural number. Obviously, $f\in N$. Let $T$ be a right transversal of $\langle g\rangle N$ in $G$. Every $s\in G$ can be written as  $s=g^knt$ for some $k\in \mathbb Z$,
$n\in N$, and $t\in T$. In this notation, we have $$f^s=f^{g^knt}=f^{nt}=(f^{t})^{n^t}\in \ll f^t\rr^N$$ since $n^{t}\in N$. This easily implies that $\ll f\rr ^G=\ll \{ f^{t} \, |\, t\in T\}\rr^N$. Using Lemma \ref{sm} and (\ref{o(g)}), we obtain
\begin{equation*}
\begin{split}
  \d\Big(N/\ll f\rr^G\Big) & \ge \d(N) - |T| \ge 1+ (\d(M) -1)|M/N| - |T|\\
  & \ge 1+ |M/N| \left(\d(M)-1 - \frac{|G:\langle g\rangle N|}{|M/N|}\right)  \\
  & \ge 1+ |M/N| \left(\d(M)-1 - \frac{|G/N|}{|\langle g\rangle N/N|\cdot |M/N|}\right)  \\
  & \ge 1+ |M/N| \left(\d(M)-1 - \frac{|G/M|}{|\langle g\rangle N/N|}\right) > 1.
\end{split}
\end{equation*}
Therefore, $\d\left(N/\ll f\rr^G\right)\ge 2$ and $G/\ll f\rr^G\in \mathcal D$.
\end{proof}

\begin{proof}[Proof of Theorem \ref{main}]
Let $F$ denote a finitely generated free group of rank at least $2$. We enumerate all elements of $F=\{ g_0=1, g_1, g_2, \ldots \} $ and construct a sequence of quotients $G_0, G_1, \ldots$ of $F$ and normal subgroups $M_i\lhd G_i$ fitting into the commutative diagram with surjective arrows
\begin{equation}\label{seq}
\begin{array}{ccccccc}
G_0 & \longrightarrow  & G_1 &  \longrightarrow & G_2 &  \longrightarrow &\ldots\\
\downarrow && \downarrow && \downarrow &&\\
G_0/M_0& \longleftarrow & G_1/M_1 & \longleftarrow & G_2/M_2 & \longleftarrow & \ldots,\\
\end{array}
\end{equation}
by the following inductive procedure.

Let $M_0=G_0=F$. Suppose that for some $k\ge 0$, we have already constructed a group $G_k$ and a subgroup $M_k\lhd G_k$ such that the following conditions hold (note that $G_0$ and $M_0$ obviously satisfy these conditions):
\begin{enumerate}
\item[(a)] $G_k\in \mathcal D$;
\item[(b)] the natural image of $g_k$ in $\widehat G_k$ has finite order;
\item[(c)] $\infty> |G_k/M_k|> k$.
\end{enumerate}

By (a) and Lemma \ref{Zd}, $G_k$ contains subgroups of arbitrarily large finite index. In particular, we can find a subgroup $L_k\lhd G_k$ such that $L_k\le M_k$ and $\infty > |G_k/L_k|> k+1$. Let $g$ denote the image of $g_{k+1}$ in $G_k$ and let $\ell =|G_k/L_k|$. By Proposition \ref{quot}, there exists $m\in \mathbb Z$ such that $G_{k+1}= G_{k}/\ll g^{\ell m}\rr ^{G_{k}}\in \mathcal D$ and the image of $g$ in $\widehat G_{k+1}$ has finite order. Let $M_{k+1}$ denote the image of $L_k$ in $G_{k+1}$. Note that $G_{k+1}/ M_{k+1} \cong G_{k}/L_k$ since $g^{\ell m}\in L_k$. Therefore, $|G_{k+1}/ M_{k+1}|=|G_k/L_k|>k+1$ and $G_{k+1}/ M_{k+1}$ naturally surjects onto $G_k/M_k$ since $L_k\le M_k$.

Thus we obtain the commutative diagram (\ref{seq}), where all horizontal arrows are surjective and all vertical arrows are natural homomorphisms. Let $G$ be the direct limit of the first row of (\ref{seq}). That is, $G=G_0/\bigcup_{k\in \mathbb N}N_k$, where $N_k$ is the kernel of the homomorphism $G_0\to G_k$ obtained by composing the first $k$ maps in the first row of (\ref{seq}). By (b), the image of $g_k$ has finite order in $\widehat G_k$ for each $k$. Since every $\widehat G_k$ naturally surjects onto $\widehat G$, the group $\widehat G$ is torsion. On the other hand, $\widehat G$ surjects onto $G_k/M_k$ for all $k$. Combining this with (c), we obtain that $\widehat G$ is infinite.
\end{proof}

\noindent \textbf{Denis Osin: } Department of Mathematics, Vanderbilt University, Nashville 37240, U.S.A.\\
E-mail: \emph{denis.v.osin@vanderbilt.edu}

\end{document}